\documentclass{amsart}[12pt]
\usepackage{amsmath,amssymb,enumerate}

\usepackage{color}

\theoremstyle{plain}
\newtheorem{thm}{Theorem}[section]
\newtheorem{lem}[thm]{Lemma}
\theoremstyle{definition}
\newtheorem{defi}[thm]{Definition}

\newtoks\by
\newtoks\paper
\newtoks\book
\newtoks\jour
\newtoks\yr
\newtoks\pages
\newtoks\vol
\newtoks\publ
\def\ota{{\hbox\vol{???}}}
\def\cLear{\by=\ota\paper=\ota\book=\ota\jour=\ota\yr=\ota
	\pages=\ota\vol=\ota\publ=\ota}
\def\endpaper{\the\by, \the\paper.
	{\it\the\jour\/} {\bf \the\vol} (\the\yr), \the\pages.\cLear}
\def\endbook{\the\by, {\it\the\book}. \the\publ.\cLear}
\def\endprep{\the\by, \the\paper. \the\jour.\cLear}
\def\name#1#2{#1 #2}
\def\et{ and }

\def\spa{\textup{span}}
\def\be{\begin{equation}}
	\def\ee{\end{equation}}

\numberwithin{equation}{section} \headheight=12pt
\begin{document}
	\title[Embeddings between variable Lebesgue \dots]{Embeddings between sequence variable Lebesgue spaces, strict and finitely strict singularity}
	
	\author[J.Lang]{J. Lang }
	\address{Department of Mathematics, The Ohio State University,
		231 West 18th Avenue, Columbus, OH 43210, USA, https://orcid.org/0000-0003-1582-7273
	\newline and \newline   Department of Mathematical Analysis,
	Faculty of Mathematics and Physics,
	Charles University,
	186 75 Praha 8, Czech Republic}
	\email{lang@math.osu.edu}
	\author[A.Nekvinda]{A. Nekvinda}
	\address{Department of Mathematics\\ Faculty of Civil Engineereng\\
		Czech Technical University\\Th\' akurova 7\\16629 Prague 6\\
		Czech Republic, https://orcid.org/0000-0001-6303-5882}
	\email{ales.nekvinda@fsv.cvut.cz}
	\date{}
	\thanks{The second author was supported by the grant P202/23-04720S of the Grant Agency of the Czech Republic}

	\maketitle
	\begin{abstract} 
		In this paper, we provide necessary and sufficient conditions under which two sequence variable Lebesgue spaces   $\ell_{p_n}$ and  $\ell_{q_n}$ are equivalent and also describe conditions under which the natural embeddings $id:\ell_{p_n} \to \ell_{q_n}$ are strictly   or finitely strictly singular.
		We also provide estimates for the Bernstein numbers of the natural embedding $id$ and show how they depend on the exponents $p_n$ and $q_n$.

	\end{abstract}
	\section{Motivation}

W. Orlicz introduced modulars for sequences with variable powers in 1931, while studying the properties of Fourier series. He defined \[\ m_{p_n}(a):=\sum |a_n|^{p_n}, \mbox{ where } a={a_n} \mbox{ and } 1\le p_n <\infty.\] In this work W. Orlicz observed (see Satz 3 and Satz 3’ in \cite{Or}) that under certain conditions on ${p_n}$, we have $m_{p_n}(a)<\infty$ if and only if $m_p(a)<\infty$ for some fixed $1\le p <\infty$.

The following year, H.R. Pitt showed \cite{Pitt} that any bounded linear operator from the sequence space $\ell_{p}$ to $\ell_{q}$, where $1\le q < p \le \infty,$ is compact. In the opposite direction (i.e. $p<q$), the situation is different and it was shown that the natural embedding $id_{pq}:\ell_{p} \to \ell_q$, when $1\le p < q \le \infty$, is non-compact and strictly singular (see \cite[Theorem 4.58]{AA book}). This result can be improved and it can be shown that $id_{pq}$ is a finitely strictly singular map, i.e. the Bernstein numbers $b_n(id_{p,q})$ decay to 0, or more precisely, we have $b_n(id_{p,q})=n^{(p-q)/pq}$ (see \cite{Pl} for proofs and historical notes). These results were probably first observed in Approximation theory (see the works by G.G. Lorentz \cite{Lor}, V.M. Tikhomirov \cite{Ti} and V.D.Milman \cite{Mil}).

This paper aims to extend and unify the existing results on the classes of sequence variable Lebesgue spaces $\l_{p_n}$, where $0<p_-\le p_n\le\infty$, in the setting of norms resp. quasi-norms. In particular, we obtain: (i) necessary and sufficient conditions for the equivalence, strict embedding or non-comparability of two sequence spaces $\l_{p_n}$ and $\l_{q_n}$, (ii) criteria for the natural embeddings $id$ to be or not to be strictly singular and finitely strictly singular (super strictly singular), along with estimates of the Bernstein numbers. This way, we generalize Orlicz’s results from 1931 and extend the known results about Bernstein numbers from sequence spaces to variable Lebesgue sequence spaces.

This paper is structured as follows. Section 2 introduces the definitions, notations and preliminary results that are essential for our main results. Section 3 presents the necessary and sufficient conditions for the embeddings between $\l_{p_n}$ spaces to hold. Section 4 investigates the Bernstein numbers of these embeddings and the criteria for them to be strictly or finitely strictly singular. Section 5, Appendix, provides some basic facts about quasi-Banach, quasi-modular and quasi-norm $\l_{p_n}$ spaces that we use throughout the paper. We include them here for the sake of completeness.

		\section{Introduction}
		
	We start by recalling some basic definitions and notations which will be used within the paper.
	\begin{defi}
		Let $X$ be a linear space and $\|.\|:X\rightarrow\mathbb{R}$ be a function. We say that $(X,\|.\|)$ is a quasi-normed space if
		\begin{enumerate}[\rm(i)]
			\item  there exists $T\ge 1$ such that $\|u+v\|\le T(\|u\|+\|v\|)$ for all $u,v\in X$,
			\item $\|\alpha u\|=|\alpha|\ \|u\|$ for each $\alpha\in\mathbb{R}$ and $u\in X$.
		\end{enumerate}
	\end{defi}
	
	\begin{defi}
		Let $(X,\|.\|)$ be a  quasi-normed space and $u_n,u\in X$.
		
		We say that $u_n$ is the Cauchy sequence  if
		$$\forall \varepsilon>0\ \exists n_0\ \in\mathbb{N} \text{ we have }\|u_n-u_m\|<\varepsilon\text{ provided }\ n,m\ge n_0\,$$
		and $u_n$ converges to $u$ if $\lim_{n\rightarrow\infty}\|u_n-u\|=0$.
	\end{defi}

\begin{defi}
		Let $(X,\|.\|)$ be a  quasi-normed space and $F\subset X$.
		We say that $F$ is closed set if for each $x_n\in F$, $x\in X$ we have $x\in F$ provided $\|x_n-x\|\rightarrow 0$.
		
	\end{defi}
	
	\begin{defi}
		A quasi-normed space $(X,\|.\|)$ is called a  quasi-Banach space  if each Cauchy sequence has a limit.
	\end{defi}

	\begin{defi}
		Let $X,Y$ be quasi-Banach spaces and $T:X\rightarrow Y$ be linear. We say that $T$ is bounded if there is $C>0$ such that
		\begin{align} \label{1.1}
			&\|Tu\|_Y\le C\|u\|_X \text{ for all } u\in X.
		\end{align}
	\end{defi}
	Remark that in the context of quasi-Banach spaces, as in Banach spaces,  the boundedness is equivalent to the continuity of $T$. It means that (\ref{1.1}) is equivalent to $\|Tu_n-Tu\|_Y\rightarrow 0$ provided $\|u_n-u\|_X\rightarrow 0$.
	
	\begin{defi}
		Let $X,Y$ be quasi-Banach spaces and $T:X\rightarrow Y$ be a linear bounded map. We say that $T$ is strictly singular if there is no subspace $E\subset X$, $\dim E=\infty$, such that $T:E\rightarrow T(E)$ is bounded and $T^{-1}:T(E)\rightarrow E$ is bounded, too.
	\end{defi}
	
	\begin{defi}
		Let $X,Y$ be quasi-Banach spaces and $T:X\rightarrow Y$ be linear and bounded. We say that $T$ is finitely strictly singular if for each $\varepsilon>0$ there exists $n\in\mathbb{N}$ such that for all subspaces $E\subset X$, $\dim(E)\ge n$ we can find $x\in E$ with $\|x\|_X=1$ and $\|T(x)\|_Y\le\varepsilon$.
	\end{defi}
	
	\begin{defi}
		Let $X,Y$ be quasi-Banach spaces and $T:X\rightarrow Y$ be linear and bounded. We define $n$-th Bernstein number by
		\begin{align*}
			&b_n(T):=\sup_{E\subset X, \dim(E)=n}\inf_{u\in E, \|u\|_X=1}\|Tu\|_Y.
		\end{align*}
	\end{defi}
	
	\begin{defi}
		Let $\mathcal{S}$ be a  set of all sequences of real numbers and $\|.\|:\mathcal{S}\rightarrow[0,\infty]$. Assume that for all $u,v\in \mathcal{S}$ and $\alpha\in\mathbb{R}$ the function $\|.\|$ satisfies:
		\begin{enumerate}[\rm(i)]
			\item  $\|u+v\|\le T(\|u\|+\|v\|)$ for some $T\ge 1$,
			\item $\|\alpha u\|=|\alpha|\ \|u\|$,
			\item $\|u\|\ge 0,$ and $\|u\|= 0$ if and only if $u=0$,
			\item $\|u\|=\|\ |u|\ \|$,
			\item if $|u|\le |v|$ then  $\|u\|\le \|v\|$,
			\item if $0\le u_n\nearrow u$ then $\|u_n\|\nearrow\|u\|$,
			\item if $\#\{i;u(i)\neq 0\}<\infty$ then $\|u\|<\infty$.
		\end{enumerate}
		Define $X:=\{u;\|u\|<\infty\}$. Then we call $X$ a quasi-Banach sequence function space. Note  that in \cite{NP} it is proved that such spaces are complete.
	\end{defi}
	
	\begin{defi}
		Let $0<p_n\le \infty$ be a sequence. Set $\mathbb{F}_p=\{n\in\mathbb{N};p_n<\infty\}$ and define  quantities $m_{p_n}$ and $\|.\|_{p_n}$ on all sequences $a=\{a_n\}$ by
		\begin{align*}
			m_{p_n}(a):=\sum_{n\in \mathbb{F}_p} |a_n|^{p_n}+\sup\{|a_n|;n\in\mathbb{N}\setminus\mathbb{F}_p\}.
		\end{align*}
		and
		\begin{align*}
			\|a\|_{p_n}:=\inf\Big\{\lambda>0;m_{p_n}\{a/\lambda\}\le 1\Big\}.
		\end{align*}
		Set $\ell_{p_n}:=\{a; \|a\|_{p_n}<\infty\}$. We will call this space a sequence variable Lebesgue space.
	\end{defi}

	Given a sequence $0<p_n\le \infty$ we denote $p_-:=\inf\{p_n; n\in\mathbb{N}\}$, $p_+:=\sup\{p_n; n\in\mathbb{N}\}$.
	\begin{thm}
		Let $p_->0$ then $\ell_{p_n}$ is a quasi-Banach sequence function space.
	\end{thm}
	\begin{proof}  It is necessary  to only establish  the quasi-triangle inequality when  $0<p_-<1$ as all the other axioms are easy and Fatou property follows from the proof of Lemma 2.5 in \cite{EN}.
		
		Assume $\|a\|_{p_n}=A, \|b\|_{p_n}=B$. Take any $\lambda>A, \mu>B$. Then
		\begin{align*}
			&\sum_{n=1}^{\infty}\Big(\frac{|a_n|}{\lambda}\Big)^{p_n}\le 1\ \ \  \text{and}\ \ \ \sum_{n=1}^{\infty}\Big(\frac{|b_n|}{\mu}\Big)^{p_n}\le 1.
		\end{align*}
		Using that $t\mapsto t^{p_-}$ is concave we obtain:
		\begin{align}
			&x^{p_-}+y^{p_-}\le 2^{1-p_-}(x+y)^{p_-}, \mbox{ for each } x,y \ge 0.\label{wiviovgjrg}
		\end{align}
		Calculate
		\begin{align*}
			&\sum_{n=1}^{\infty}\left(\frac{|a_n+b_n|}{2^{\frac{1}{p_-}-1}(\lambda+\mu)}\right)^{p_n}\le
			\sum_{n=1}^{\infty}\frac{1}{2^{p_n(\frac{1}{p_-}-1)}}\left(\frac{|a_n|+|b_n|}{(\lambda+\mu)}\right)^{p_n}\\
			&\le\frac{1}{2^{1-p_-}}\sum_{n=1}^{\infty} \left[ \left(\frac{|a_n|}{\lambda}\right)^{p_n}\left(\frac{\lambda}{\lambda+\mu}\right)^{p_n}+
			\left(\frac{|b_n|}{\lambda}\right)^{p_n}\left(\frac{\mu}{\lambda+\mu}\right)^{p_n} \right]\\
			&\le\frac{1}{2^{1-p_-}}\sum_{n=1}^{\infty} \left[ \left(\frac{|a_n|}{\lambda}\right)^{p_n}\left(\frac{\lambda}{\lambda+\mu}\right)^{p_-}+
			\left(\frac{|b_n|}{\lambda}\right)^{p_n}\left(\frac{\mu}{\lambda+\mu}\right)^{p_-} \right] \\
			&\le\frac{1}{2^{1-p_-}}\left[\left(\frac{\lambda}{\lambda+\mu}\right)^{p_-}+\left(\frac{\mu}{\lambda+\mu}\right)^{p_-}\right]\\
			&\overset{\text{\eqref{wiviovgjrg}}}{\le}\frac{1}{2^{1-p_-}}2^{1-p_-}\left(\frac{\lambda+\mu}{\lambda+\mu}\right)^{p_-}=1.
		\end{align*}
		Then we have  $\|a+b\|_{p_n}\le 2^{\frac{1}{p_-}-1}(\lambda+\mu)$, and by letting $\lambda\rightarrow A, \mu\rightarrow B$ we obtain
		\begin{align} \label{++}
			\|a+b\|_{p_n}\le 2^{\frac{1}{p_-}-1}(\|a\|_{p_n}+\|b\|_{p_n}).
		\end{align}
	\end{proof}

	\section{embeddings}
	
	In this section, we consider  $0<p_-\le p_n\le q_n\le \infty$. We start by recalling  the well-known theorem (can be obtained from \cite{KR}). 
	\begin{thm}
		Assume $0<p_-\le p_n\le q_n$. Then the embedding $id: \ell_{p_n}\hookrightarrow \ell_{q_n}$
		holds.
	\end{thm}

	In \cite{N} (see Lemma 4.1 and Lemma 4.2) we can find the following statement:
	\begin{lem}
		Let $1\le p_n\le q_n\le q_+<\infty$. Then the following condition are equivalent.
		\begin{enumerate}[\rm(i)]
			\item there is $c<1$ such that $\sum c^{\frac{q_n}{q_n-p_n}}<\infty$
			\item $\ell_{q_n}\hookrightarrow \ell_{p_n}$.
		\end{enumerate}
	\end{lem}
	
	\begin{thm}\label{devkivvjknh}
		Assume $0<p_-\le p_n$. Let $\ell_{p_n}$ be a sequence variable Lebesgue space. Assume that there is $0<c<1$ such that
		\begin{equation}
			\sum_{n=1}^{\infty}c^{p_n}=:M<\infty. \label{(*)}
		\end{equation}
		Then $\ell_{p_n}$ and $\ell_{\infty}$ coinside. Moreover inequalities
		\begin{align*}
			&\|a\|_{\infty}\le \|a\|_{p_n}\le{\frac{\max\{1,M\}^{1/{p_-}}}{c}}\|a\|_{\infty}
		\end{align*}
		hold for all sequences $a=\{a_n\}$.
	\end{thm}
	\begin{proof} Note that (\ref{(*)}) gives  $\lim_{n\to \infty} p_n=\infty$. Let $a=\{a_n \} \in\ell_{p_n}$. Take $\lambda>\|a\|_{{p_n}}$. Then
		$$
		\sum_{n=1}^{\infty}\Big(\frac{|a_n|}{\lambda}\Big)^{p_n}<1\ \Rightarrow\ \sup_n\frac{|a_n|}{\lambda}\leq 1\ \Rightarrow\ \|a\|_{{\infty}}\le \lambda.
		$$
		This gives $\|a\|_{{\infty}}\le \|a\|_{{p_n}}$.
		
		Let now $\|a\|_{{\infty}}<\infty$. Take $\lambda=\frac{\max\{1,M\}^{1/{p_-}}\|a\|_{{\infty}}}{c}$. Since $|a_n|\le \|a\|_\infty$ for all $n\in\mathbb{N} $ we obtain
			\begin{align*}
				&\sum_{n=1}^{\infty}\Big(\frac{|a_n|}{\lambda}\Big)^{p_n}= \sum_{n=1}^{\infty}\Big(\frac{|a_n|c}{\max\{1,M\}^{1/{p_-}}\|a\|_{{\infty}}}\Big)^{p_n}\\
				&\le\sum_{n=1}^{\infty}\frac{c^{p_n}}{\max\{1,M\}^{p_n/{p_-}}}\le \frac{1}{\max\{1,M\}}\sum_{n=1}^{\infty}c^{p_n}=\frac{M}{\max\{1,M\}}\le 1
		\end{align*}
		and so,
		$$
		\|a\|_{{p_n}}\le \frac{\max\{1,M\}^{1/{p_-}}}{c}\|a\|_{{\infty}}.
		$$
	\end{proof}
	
	\begin{thm}\label{wovjvgj}
		Assume $0<p_-\le p_n$. Let $\ell_{p_n}$ be a sequence variable Lebesgue space. Assume
		$$
		\sum_{n=1}^{\infty}c^{p_n}=\infty , \mbox{ for all } 0<c<1.
		$$
		Then $\ell_{p_n}$ and $\ell_{\infty}$ do not coincide and the norms are not equivalent, i.e. $l_{p_n} \hookrightarrow l_{\infty}$ and $l_{\infty} \not \hookrightarrow l_{p_n}$.
	\end{thm}
	\begin{proof}It suffices to find any sequence $a=\{a_n\}$ such that $a\in \ell_\infty\setminus \ell_{p_n}$. We take $a_n=1$ for all $n$. Then $a\in\ell_\infty$ and by the assumption $a\notin\ell_{p_n}$.
	\end{proof}

	\begin{lem}\label{flvhohho}
		Let $0<p_-\le p_n<q_n<\infty$ for all $n$. Assume that we have for some $0<c<1$
		\begin{align*}
			&M:=\sum_{n=1}^\infty c^{\frac{1}{\frac{1}{p_n}-\frac{1}{q_n}}}<\infty.
		\end{align*}
		Then $\ell_{q_n}\hookrightarrow\ell_{p_n}$. Moreover the norm of this embedding can be majorized by $\frac{4 \max\{1,M\}^{1/p_-}}{c}$.
	\end{lem}
	\begin{proof}
		Let $a=\{a_n\}\in\ell_{q_n}$, $\|a\|_{q_n}=1$. Then we have
		\begin{align*}
			&\sum_{n=1}^{\infty}\Big(\frac{|a_n|}{\lambda}\Big)^{q_n}\le 1
		\end{align*}
		for all $\lambda>1$ . Fix such $\lambda$.
		Set $K:=1/c$ and then by H\"older inequality (see \cite{KR}, Theorem 2.1) we obtain
		\begin{align*}
			&\sum_{n=1}^{\infty}\Big(\frac{|a_n|}{K\lambda}\Big)^{p_n}\le 4\Big\|\Big(\frac{|a_n|}{\lambda}\Big)^{p_n}\Big\|_{\frac{q_n}{p_n}}.\|c^{p_n}\|_{(\frac{q_n}{p_n})'}.
		\end{align*}
		
		Estimate $\Big\|\Big(\frac{|a_n|}{\lambda}\Big)^{p_n}\Big\|_{\frac{q_n}{p_n}}$. Since
		\begin{align*}
			&\sum_{n=1}^\infty \left(\frac{\big(\frac{|a_n|}{\lambda}\big)^{p_n}}{1}\right)^{\frac{q_n}{p_n}}=\sum_{n=1}^\infty \Big(\frac{|a_n|}{\lambda}\Big)^{q_n}\le 1
		\end{align*}
		we have $\Big\|\Big(\frac{|a_n|}{\lambda}\Big)^{p_n}\Big\|_{\frac{q_n}{p_n}}\le 1$.
		
		Estimate now $\|\{c^{p_n} \} \|_{(\frac{q_n}{p_n})'}$. Clearly
		\begin{align*}
			&\sum_{n=1}^\infty (c^{p_n})^{(\frac{q_n}{p_n})'}=\sum_{n=1}^\infty (c^{p_n})^{\frac{q_n}{q_n-p_n}}=\sum_{n=1}^\infty c^{\frac{1}{\frac{1}{p_n}-\frac{1}{q_n}}}=M<\infty.
		\end{align*}
		Set $R:=\max\{1,M\}$. Then
		\begin{align*}
			&1=\sum_{n=1}^\infty \frac{1}{R}(c^{p_n})^{(\frac{q_n}{p_n})'}\ge \sum_{n=1}^\infty \Big(\frac{1}{R}\Big)^{(\frac{q_n}{p_n})'}(c^{p_n})^{(\frac{q_n}{p_n})'}=\sum_{n=1}^\infty \Big(\frac{c^{p_n}}{R}\Big)^{(\frac{q_n}{p_n})'}.
		\end{align*}
		It gives us $\|\{ c^{p_n} \}\|_{(\frac{q_n}{p_n})'}\le R$. At last we have proved the implication
		\begin{align*}
			&\sum_{n=1}^{\infty}\Big(\frac{|a_n|}{\lambda}\Big)^{q_n}\le 1\ \Longrightarrow\ \sum_{n=1}^{\infty}\Big(\frac{c|a_n|}{\lambda}\Big)^{p_n}\le 4R.
		\end{align*}
		Then
		\begin{align*}
			&\sum_{n=1}^{\infty}\Big(\frac{c|a_n|}{\lambda (4R)^{1/p_-}}\Big)^{p_n}\le\sum_{n=1}^{\infty}\Big(\frac{c|a_n|}{\lambda (4R)^{1/p_n}}\Big)^{p_n}=\frac{1}{4R}\sum_{n=1}^{\infty}\Big(\frac{c|a_n|}{\lambda }\Big)^{p_n}\le 1.
		\end{align*}
		Since $\lambda>1$ is arbitrary we obtain
		\begin{align*}
			&\|a\|_{p_n}\le 4\frac{\max\{1,M\}^{1/p_-}}{c}\|a\|_{q_n}
		\end{align*}
		This implies the embedding $\ell_{q_n}\hookrightarrow\ell_{p_n}$.
	\end{proof}
	
In the next we set $1/q_n=0$ for $q_n=\infty$.	
	\begin{lem}\label{dkdjkvjkv}
		Let $0<p_-\le p_n<q_n\le\infty$ for all $n$. Assume that we have for some $0<c<1$
		\begin{align*}
			&M:=\sum_{n=1}^\infty c^{\frac{1}{\frac{1}{p_n}-\frac{1}{q_n}}}<\infty.
		\end{align*}
		Then $\ell_{q_n}\hookrightarrow\ell_{p_n}$. Moreover the norm of this embedding can be majorized by  $5\cdot  2^{1/p_-}\frac{\max\{1,M\}^{1/p_-}}{c}$.
	\end{lem}
	\begin{proof}
		Set $\mathbb{N}_q=\{n; q_n<\infty\}$. Let $\|a\|_{q_n}= 1$. Clearly, by Lemma \ref{flvhohho} and Theorem \ref{devkivvjknh} we have
		\begin{align*}
			&\|a\|_{p_n}\le 2^{1/p_-}\big(\|a\chi_{\mathbb{N}_q}\|_{p_n}+\|a\chi_{\mathbb{N}\setminus\mathbb{N}_q}\|_{p_n}\big)\\
			&\le 2^{1/p_-}\Big(\frac{4 \max\{1,M\}^{1/p_-}}{c}+\frac{\max\{1,M\}^{1/p_-}}{c}\Big)\|a\|_{q_n}\\
            &=5 \cdot 2^{1/p_-}\frac{\max\{1,M\}^{1/p_-}}{c}\|a\|_{q_n}.
		\end{align*}
	\end{proof}
	
	\begin{thm}\label{eflkkbtnbk}
		Let $0<p_-\le p_n, q_n\le\infty$ for all $n$. Denote $\mathbb{A}=\{n;p_n<q_n\}$. Assume that we have for some $0<c<1$
		\begin{align*}
			&M:=\sum_{n\in\mathbb{A}}c^{\frac{1}{\frac{1}{p_n}-\frac{1}{q_n}}}<\infty.
		\end{align*}
		Then $\ell_{q_n}\hookrightarrow\ell_{p_n}$ and the norm of this embedding is majorized by
		$2^{1/p_-}\Big((5\cdot 2^{1/p_-}+1)\frac{\max\{1,M\}^{1/p_-}}{c}+1\Big)$.
	\end{thm}
	\begin{proof}
		Let $\|a\|_{q_n}= 1$. Then by Lemma \ref{dkdjkvjkv} we obtain
		\begin{align*}
			&\|a\|_{p_n}\le 2^{1/p_-}\big(\|a\chi_{\mathbb{A}}\|_{p_n}+\|a\chi_{\mathbb{N}\setminus\mathbb{A}}\|_{p_n}\big)\\
			&\le 2^{1/p_-}\Big(5\cdot 2^{1/p_-}\frac{\max\{1,M\}^{1/p_-}}{c}+\frac{\max\{1,M\}^{1/p_-}}{c}+1\Big)\|a\|_{q_n}.
		\end{align*}
	\end{proof}

	\begin{lem}\label{dcvhdlvbh}
		Let $0<p_-\le  p_n<q_n<\infty$ for all $n$ and $K>1$. Assume that we have for all $0<c<1$
		\begin{align*}
			&\sum_{n=1}^{\infty}c^{\frac{1}{\frac{1}{p_n}-\frac{1}{q_n}}}=\infty.
		\end{align*}
		Then there exists a sequence $a_n>0$ with
		\begin{align*}
			&\sum_{n=1}^{\infty}a_n^{q_n}\le 1\ \text{ and }\ \sum_{n=1}^{\infty}\Big(\frac{a_n}{K}\Big)^{p_n}=\infty.
		\end{align*}
	\end{lem}
	\begin{proof}
		Set $\alpha:=2^{\frac{1}{p_-}}$, and note that $\alpha>1$.
		Given an increasing  sequence of positive integers $n_k$ (set $n_0=0$) denote $\mathbb{N}_{k+1}=\{n_k+1,n_k+2,\dots,n_{k+1}\}$.
		Construct now inductively numbers $c_k>0$ and $n_k$ such that
		\begin{align*}
			&0<c_k\le\frac{1}{\alpha^kK}\ \text{ and }\ \sum_{n\in \mathbb{N}_k}c_k^{\frac{1}{\frac{1}{p_n}-\frac{1}{q_n}}}=1.
		\end{align*}
		For $k=1$ we find $n_1$ sufficiently large such that $\sum_{n\in\mathbb{N}_1}(\frac{1}{2K})^{\frac{1}{\frac{1}{p_n}-\frac{1}{q_n}}}\ge 1$. Clearly we can take $0<c_1\le \frac{1}{\alpha K}$ such that
		$\sum_{n\in\mathbb{N}_1}(c_1)^{\frac{1}{p_n}-\frac{1}{q_n}}=1$. Assume now that we have constructed $c_1,c_2, \dots,c_k$ and $n_1,n_2,\dots,n_k$. Find now $n_{k+1}$ such that
		$\sum_{n\in\mathbb{N}_{k+1}}(\frac{1}{\alpha^{k+1}K})^{\frac{1}{\frac{1}{p_n}-\frac{1}{q_n}}}\ge 1$. Take $0<c_{k+1}\le \frac{1}{c^{k+1}K}$ such that $\sum_{n\in\mathbb{N}_{k+1}}c_{k+1}^{\frac{1}{\frac{1}{p_n}-\frac{1}{q_n}}}= 1$.

		Define $a_n$ by
		$\big(\frac{a_n}{K}\big)^{p_n}:=c_k^{\frac{1}{\frac{1}{p_n}-\frac{1}{q_n}}}$, $n\in\mathbb{N}_k$. Then
		\begin{align*}
			&\sum_{n=1}^{\infty}\Big(\frac{a_n}{K}\Big)^{p_n}=\sum_{k=1}^{\infty}\sum_{n\in\mathbb{N}_k} c_k^{\frac{1}{\frac{1}{p_n}-\frac{1}{q_n}}}=\sum_{k=1}^{\infty}1=\infty.
		\end{align*}
		Further due to obvious identity,
		\begin{align*}
			&\frac{\frac{q_n}{p_n}}{\frac{1}{p_n}-\frac{1}{q_n}}=q_n+\frac{1}{\frac{1}{p_n}-\frac{1}{q_n}}
		\end{align*}
		we obtain
		\begin{align*}
			&\sum_{n=1}^{\infty}a_n^{q_n}=\sum_{n=1}^{\infty}\Big( K\frac{a_n}{K}\Big)^{q_n}=\sum_{n=1}^{\infty} K^{q_n}\Big(\frac{a_n}{K}\Big)^{q_n}\\
			&=\sum_{k=1}^{\infty}\sum_{n\in\mathbb{N}_k} K^{q_n} c_k^{\frac{\frac{q_n}{p_n}}{\frac{1}{p_n}-\frac{1}{q_n}}}=\sum_{k=1}^{\infty}\sum_{n\in\mathbb{N}_k} K^{q_n} c_k^{q_n}c_k^{\frac{1}{\frac{1}{p_n}-\frac{1}{q_n}}}\\
			&\le \sum_{k=1}^{\infty}\sum_{n\in\mathbb{N}_k} K^{q_n}\Big(\frac{1}{\alpha^k K}\Big)^{q_n}c_k^{\frac{1}{\frac{1}{p_n}-\frac{1}{q_n}}}= \sum_{k=1}^{\infty}\sum_{n\in\mathbb{N}_k} \Big(\frac{1}{\alpha^k}\Big)^{q_n}c_k^{\frac{1}{\frac{1}{p_n}-\frac{1}{q_n}}}  \\
			&\le \sum_{k=1}^{\infty}\sum_{n\in\mathbb{N}_k} \Big(\frac{1}{\alpha^k}\Big)^{p_-}c_k^{\frac{1}{\frac{1}{p_n}-\frac{1}{q_n}}}=\sum_{k=1}^{\infty}\sum_{n\in\mathbb{N}_k} \Big(\frac{1}{\alpha^{p_-}}\Big)^k c_k^{\frac{1}{\frac{1}{p_n}-\frac{1}{q_n}}}\\
			&= \sum_{k=1}^{\infty}\sum_{n\in\mathbb{N}_k} \frac{1}{2^k}c_k^{\frac{1}{\frac{1}{p_n}-\frac{1}{q_n}}}=\sum_{k=1}^{\infty}\frac{1}{2^k}\sum_{n\in\mathbb{N}_k} c_k^{\frac{1}{\frac{1}{p_n}-\frac{1}{q_n}}}=\sum_{k=1}^{\infty}\frac{1}{2^k}=1.
		\end{align*}

	\end{proof}
	
	\begin{lem}\label{wdovhwrovg}
		Let $0<p_-\le  p_n<q_n<\infty$ for all $n$ and $\varepsilon<1$, $K>1$. Assume that we have for all $0<c<1,$
		\begin{align*}
			&\sum_{n=1}^{\infty}c^{\frac{1}{\frac{1}{p_n}-\frac{1}{q_n}}}=\infty.
		\end{align*}
		Then there exists a sequence $a_n>0$ with
		\begin{align*}
			&\sum_{n=1}^{\infty}\Big(\frac{ a_n}{\varepsilon}\Big)^{q_n}\le 1\ \text{ and }\ \sum_{n=1}^{\infty}\Big(\frac{a_n}{K}\Big)^{p_n}=\infty.
		\end{align*}
	\end{lem}
	\begin{proof}
		Set $L=\frac{K}{\varepsilon}$. By Lemma \ref{dcvhdlvbh} we have a sequence $b_n$ such that
		\begin{align*}
			&\sum_{n=1}^{\infty}b_n^{q_n}\le 1\ \text{ and }\ \sum_{n=1}^{\infty}\Big(\frac{b_n}{L}\Big)^{p_n}=\infty.
		\end{align*}
		Setting $a_n=\varepsilon b_n$ we have the required sequence.
	\end{proof}
	
	\begin{lem}\label{wdklefjklbn}
		Let $0<p_-\le  p_n<q_n<\infty$ for all $n$. Assume that
		\begin{align}
			&\sum_{n=1}^{\infty}c^{\frac{1}{\frac{1}{p_n}-\frac{1}{q_n}}}=\infty\ \text{ for all }\ c>0. \label{Ctvercek}
		\end{align}
		Then there exists a sequence $a_n>0$ such that for all $K>1$ we have
		\begin{align*}
			&\sum_{n=1}^{\infty}a_n^{q_n}\le 1\ \text{ and }\ \sum_{n=1}^{\infty}\Big(\frac{a_n}{K}\Big)^{p_n}=\infty.
		\end{align*}
	\end{lem}
	\begin{proof}
		Take an increasing sequence of non-negative integers $n_k$ (set $n_0=0$) and denote $\mathbb{A}_{k}=\{n_{k-1}+1,n_{k-1}+2,\dots,n_{k}\}$. Due to assumption (\ref{Ctvercek}), we can find $n_k$ such that
		\begin{align*}
			&\sum_{n\in \mathbb{A}_k}\Big(\frac{1}{k}\Big)^{\frac{1}{\frac{1}{p_n}-\frac{1}{q_n}}}\ge 1.
		\end{align*}
		Split $\mathbb{N}$ into infinite many infinite sets $\mathbb{S}_k$ and define $\mathbb{N}_k:=\bigcup_{j\in \mathbb{S}_k}\mathbb{A}_j$. Let $c>0$. Set $k_0=\big[\frac{1}{c}\big]+1$. Then for $k\ge k_0$ we have $\frac{1}{k}\le \frac{1}{k_0}\le c$. Then we obtain
		\begin{align*}
			&\sum_{n\in\mathbb{N}_k}c^{\frac{1}{\frac{1}{p_n}-\frac{1}{q_n}}}=\sum_{k=1}^\infty \sum_{j\in \mathbb{A}_j}c^{\frac{1}{\frac{1}{p_n}-\frac{1}{q_n}}}\\
			&=\sum_{k=1}^{k_0-1} \sum_{n\in \mathbb{A}_j}c^{\frac{1}{\frac{1}{p_n}-\frac{1}{q_n}}}+\sum_{k=k_0}^\infty \sum_{n\in \mathbb{A}_j}c^{\frac{1}{\frac{1}{p_n}-\frac{1}{q_n}}}\\
			&\ge \sum_{k=1}^{k_0-1} \sum_{n\in \mathbb{A}_j}c^{\frac{1}{\frac{1}{p_n}-\frac{1}{q_n}}}+\sum_{k=k_0}^\infty \sum_{n\in \mathbb{A}_j}\Big(\frac{1}{k}\Big)^{\frac{1}{\frac{1}{p_n}-\frac{1}{q_n}}}\\
			&\ge  \sum_{k=1}^{k_0-1} \sum_{n\in \mathbb{A}_j}c^{\frac{1}{\frac{1}{p_n}-\frac{1}{q_n}}}+\sum_{k=k_0}^\infty 1=\infty.
		\end{align*}
		
		We have constructed  pairwise disjoint infinite sets $\mathbb{N}_k\subset \mathbb{N}$, $k=1,2,\dots$ such that $\mathbb{N}=\bigcup_{k=1}^\infty \mathbb{N}_k$, and for each $c>0$ and $k\in \mathbb{N}$ we have
		\begin{align*}
			&\sum_{n\in\mathbb{N}_k}c^{\frac{1}{\frac{1}{p_n}-\frac{1}{q_n}}}=\infty.
		\end{align*}
		
		Set $\alpha:=2^{\frac{1}{p_-}}$. Due to Lemma \ref{wdovhwrovg} there are sequences $b^{(k)}=\{b_n^{(k)}\}_{n\in \mathbb{N}_k}$, such that
		\begin{align*}
			&\sum_{n\in \mathbb{N}_k}(\alpha^k b_n^{(k)})^{q_n}\le 1\ \text{ and }\ \sum_{n\in \mathbb{N}_k}\Big(\frac{b_n^{(k)}}{k}\Big)^{p_n}=\infty.
		\end{align*}
		Clearly
		\begin{align*}
			&\sum_{n\in \mathbb{N}_k}2^k(b_n^{(k)})^{q_n}=\sum_{n\in \mathbb{N}_k}(\alpha^k)^{p_-}(b_n^{(k)})^{q_n}\le \sum_{n\in \mathbb{N}_k}(\alpha^k)^{q_n}(b_n^{(k)})^{q_n}=\sum_{n\in \mathbb{N}_k}(\alpha^k b_n^{(k)})^{q_n}\le 1
		\end{align*}
		and so
		\begin{align*}
			&\sum_{n\in \mathbb{N}_k}(b_n^{(k)})^{q_n}\le \frac{1}{2^k}.
		\end{align*}
		Extend $b^{(k)}:=\{b_n^{(k)}\}$ by zero outside of $\mathbb{N}_k$. Define a sequence $a:=\{a_n\}$ by
		\begin{align*}
			&a=\sum_{k\in \mathbb{N}}b^{(k)}\chi_{\mathbb{N}_k}.
		\end{align*}
		Then
		\begin{align*}
			&\sum_{n=1}^\infty a_n^{q_n}=\sum_{k=1}^\infty \sum_{n\in \mathbb{N}_k}(b_n^{(k)})^{q_n}\le \sum_{k=1}^\infty \frac{1}{2^k}=1.
		\end{align*}
		Choose $K>1$. Find $s\in\mathbb{N}$ such that $s>K$. Then
		\begin{align*}
			&\sum_{n=1}^\infty \Big(\frac{a_n}{K}\Big)^{p_n}=\sum_{k=1}^\infty \sum_{n\in \mathbb{N}_k}\Big(\frac{b_n^{(k)}}{K}\Big)^{q_n}\ge\sum_{n\in \mathbb{N}_s}\Big(\frac{b_n^{(s)}}{K}\Big)^{q_n}\ge\sum_{n\in \mathbb{N}_s}\Big(\frac{b_n^{(s)}}{s}\Big)^{q_n}=\infty.
		\end{align*}
	\end{proof}
	
	\begin{lem}\label{owkdvlvjlkvjk}
		Let $0<p_-\le  p_n<q_n<\infty$ for all $n$. Assume that for all $0<c<1$ we have
		\begin{align*}
			&\sum_{n=1}^{\infty}c^{\frac{1}{\frac{1}{p_n}-\frac{1}{q_n}}}=\infty.
		\end{align*}
		Then $\ell_{q_n}$ is not embedded into $\ell_{p_n}$, i.e. $\ell_{p_n} \hookrightarrow \ell_{q_n}$ and $\ell_{q_n} \not \hookrightarrow l_{p_n}$.
	\end{lem}
	\begin{proof}
		By Lemma \ref{wdklefjklbn} we can find a sequence $a=\{a_n\}$ such that for all $K>1$ we have
		\begin{align*}
			&\sum_{n=1}^{\infty}a_n^{q_n}\le 1\ \text{ and }\ \sum_{n=1}^{\infty}\Big(\frac{a_n}{K}\Big)^{p_n}=\infty.
		\end{align*}
		It immediately implies
		\begin{align*}
			&\|a\|_{q_n}\le 1\ \text{ and }\ \|a\|_{p_n}=\infty.
		\end{align*}
		So $a\in \ell_{q_n}\setminus \ell_{p_n}$.
	\end{proof}
	
	\begin{thm}\label{xkclvlksjkjk}
		Let $0<\min (p_-, q_-) \le p_n, q_n\le\infty$ for all $n$. Denote $\mathbb{A}=\{n;p_n<q_n\}$, and let $\# \mathbb{A}=\infty$. Assume that
		\begin{align*}
			&\sum_{n\in\mathbb{A}}c^{\frac{1}{\frac{1}{p_n}-\frac{1}{q_n}}}=\infty \mbox{ for all } 0<c< 1.
		\end{align*}
		Then $\ell_{q_n}$ is not embedded into $\ell_{p_n}$.
	\end{thm}
	\begin{proof}Set $\mathbb{F}_q=\{n;q_n<\infty\}$. Then for each $c>0$ we have
		\begin{align*}
			&\sum_{n\in\mathbb{A}}c^{\frac{1}{\frac{1}{p_n}-\frac{1}{q_n}}}=
			\sum_{n\in\mathbb{A}\cap\mathbb{F}_q}c^{\frac{1}{\frac{1}{p_n}-\frac{1}{q_n}}}+
			\sum_{n\in\mathbb{A}\cap(\mathbb{N}\setminus\mathbb{F}_q)}c^{\frac{1}{\frac{1}{p_n}-\frac{1}{q_n}}}=:I(c)+J(c)=\infty.
		\end{align*}
		For fixed $c$ either $I(c)=\infty$ or $J(c)=\infty$. Denote $I:=\{c, I(c)=\infty\}$, $J:=\{c, J(c)=\infty\}$. So $I\cup J=(0,1)$. It gives us that either $I$ or $J$ has $0$ as its infimum. Let $c_1<c_2$. Then clearly $I(c_1)\le I(c_2)$ and similarly $J(c_1)\le J(c_2)$. It implies that either $I=(0,1)$ or $J=(0,1)$.
		
		Assume first that $I=(0,1)$. Then by Lemma \ref{owkdvlvjlkvjk} there is a sequence $a\in \ell_{q_n}(\mathbb{A}\cap\mathbb{F}_q)\setminus\ell_{p_n}(\mathbb{A}\cap\mathbb{F}_q)$ and consequently $\ell_{q_n}$ is not embedded into $\ell_{p_n}$.
		
		It remains the case $J=(0,1)$. Remark that $q_n=\infty$ in this case. Then by Lemma \ref{wovjvgj} there is a sequence $a\in \ell_{q_n}(\mathbb{A}\cap(\mathbb{N}\setminus\mathbb{F}_q))\setminus\ell_{p_n}(\mathbb{A}\cap(\mathbb{N}\setminus\mathbb{F}_q))$ and consequently $\ell_{q_n}$ is not embedded into $\ell_{p_n}$.
	\end{proof}
	
	The following theorem is an easy consequence of Theorem \ref{eflkkbtnbk} and Theorem \ref{xkclvlksjkjk}.
	\begin{thm}\label{dkjrekjgrkjjj}
		Let $0<\min(p_-, q_-)\le p_n, q_n\le\infty$ for all $n$. Denote $\mathbb{A}=\{n;p_n\neq q_n\}$. Then the following conditions are equivalent:
		\begin{enumerate}[\rm(i)]
			\item There is $0<c<1$ such that $\sum_{n\in\mathbb{A}}c^{\frac{1}{|\frac{1}{p_n}-\frac{1}{q_n}|}}<\infty$,
			\item The norms in $\ell_{q_n}$ and in $\ell_{p_n}$ are equivalent.
		\end{enumerate}
	\end{thm}
	
	\section{Strict singularity and Bernstein numbers}

	\begin{thm}\label{eovvjbjpjpj}
		Let $0<p_-\le  p_n<q_n<\infty$ for all $n$. Assume that there exists a infinite subset $\mathbb{S}\subset\mathbb{N}$ and $0<c<1$ such that
		\begin{align} \label{***}
			&\sum_{n\in \mathbb{S}}c^{\frac{1}{\frac{1}{p_n}-\frac{1}{q_n}}}<\infty.
		\end{align}
		Then $id:\ell_{p_n}\hookrightarrow \ell_{q_n}$ is not strictly singular.
	\end{thm}
	\begin{proof}
		Set
		\begin{align*}
			&X:=\{\{a_n\}; a_n=0\ \text{for}\ n\notin\mathbb{S}\}.
		\end{align*}
		Denote $X_p:=(X,\|.\|_{p_n})$, $X_q:=(X,\|.\|_{q_n})$. Then $X_p\subset \ell_{p_n}$, $X_q\subset \ell_{q_n}$ are infinite dimensional subspaces and $id:X_p\hookrightarrow X_q$ is bounded. By Theorem \ref{flvhohho} $id:X_q\hookrightarrow X_q$ is bounded and so $id:\ell_{p_n}\hookrightarrow \ell_{q_n}$ is not strictly singular.
	\end{proof}
	In the next Lemma we replace (\ref{***})
by a condition which is better to check.	\begin{lem}
		Let $0<p_-\le  p_n<q_n<\infty$ for all $n$. Assume
		\begin{align*}
			&\limsup_{n\rightarrow\infty}\frac{1}{\frac{1}{p_n}-\frac{1}{q_n}}=\infty.
		\end{align*}
		Then $id:\ell_{p_n}\hookrightarrow \ell_{q_n}$ is not strictly singular.
	\end{lem}
	\begin{proof}By the assumption there exists a infinite subset $\mathbb{S}:=\{n_1,n_2,n_3,\dots\}\subset\mathbb{N}$
		such that
		\begin{align*}
			&\frac{1}{\frac{1}{p_{n_k}}-\frac{1}{q_{n_k}}}\ge k.
		\end{align*}
		Set
		\begin{align*}
			&X:=\{\{a_n\}; a_n=0\ \text{for}\ n\notin\mathbb{S}\}, \mbox{ and } X_p, X_q \mbox{ as in the previous proof}.
		\end{align*}
		Then
		\begin{align*}
			&\sum_{n\in \mathbb{S}}\Big(\frac{1}{2}\Big)^{\frac{1}{\frac{1}{p_n}-\frac{1}{q_n}}}=\sum_{k=1}^{\infty}\Big(\frac{1}{2}\Big)^{\frac{1}{\frac{1}{p_{n_k}}-\frac{1}{q_{n_k}}}}\le
			\sum_{k=1}^{\infty}\Big(\frac{1}{2}\Big)^{k }<\infty.
		\end{align*}
		By Theorem \ref{eovvjbjpjpj} $id:X_p\hookrightarrow X_q$ and the inverse identity are both bounded. Thus the embedding $id:\ell_{p_n}\hookrightarrow \ell_{q_n}$ is not strictly singular.
	\end{proof}

Now we introduce a lemma concerning $n$-dimensional subspaces of sequence spaces.
\begin{lem}[see Lemma 4 in \cite{Pl}]\label{plichko}
Let $X$ be some linear space of sequences 
vanishing at infinity. Then every subspace $E_n \subset X$ , $\dim {E_n} = n$, contains a
sequence $x(i)$ such that $\max\{|x(i)|, i \in \mathbb{N}\}$ is attained in at least n points.
\end{lem}
	
	\begin{lem}\label{kdwckdcckvckk}
		Let $0<p_-\le  p_n<q_n\le q_+<\infty$ for all $n$. Assume that there is $\alpha$ such that for all $n$
		\begin{align*}
			&q_n-p_n\ge \alpha>0.
		\end{align*}
		Then the embedding $id:\ell_{p_n}\hookrightarrow \ell_{q_n}$ satisfies
		\begin{align*}
			&b_n(id)\le n^{-\frac{\alpha}{(q_+-\alpha)q_+}}.
		\end{align*}
	\end{lem}
	\begin{proof} Let $X_n\subset \ell_{p_n}$ be an $n$-dimensional subspace. Due to Lemma \ref{plichko} we can take $a\in X_n$, $\|a\|_{p_n}=1$ such that
		$A:=\max\{|a_i|;i\in\mathbb{N}\}=|a_{k_1}|=|a_{k_2}|=\dots=|a_{k_n}|$. Clearly $A\le 1$. Then
		\begin{align*}
			&\sum_{k=1}^\infty |a_k|^{q_k}=\sum_{k=1}^\infty |a_k|^{p_k}|a_k|^{q_k-p_k}\le \sum_{k=1}^\infty |a_k|^{p_k}|a_k|^{\alpha}\le A^\alpha \sum_{k=1}^\infty |a_k|^{p_k}=A^\alpha.
		\end{align*}
		Moreover
		\begin{align*}
			&nA^{p_+}=\sum_{i=1}^n A^{p_+}=\sum_{i=1}^n A^{p_{k_i}}\le\sum_{i=1}^\infty |a_i|^{p_{i}}=1.
		\end{align*}
		So
		\begin{align*}
			& A\le n^{-\frac{1}{p_+}}.
		\end{align*}
		Thus
		\begin{align*}
			&\sum_{i=1}^n \Big(\frac{|a_i|}{A^{\frac{\alpha}{q_+}}}\Big)^{q_i}\le \sum_{i=1}^n \Big(\frac{|a_i|}{A^{\frac{\alpha}{q_i}}}\Big)^{q_i} \le  \sum_{i=1}^n\frac{|a_i|^{q_i} }{A^\alpha}\le 1          .
		\end{align*}
		and cosequently $\|a\|_{q_n}\le A^{\frac{\alpha}{q_+}}$. It gives $b_n(id)\le A^{\frac{\alpha}{q_+}}\le n^{-\frac{\alpha}{p_+ q_+}}\le n^{-\frac{\alpha}{(q_+-\alpha)q_+}}$.
		
	\end{proof}
	
	\begin{thm}\label{nvjgkdocnvkvevjfe}
		Let $0<p_-\le  p_n<q_n<\infty$ for all $n$. Assume
		\begin{align*}
			&\limsup_{n\rightarrow\infty}\frac{1}{\frac{1}{p_n}-\frac{1}{q_n}}<\infty.
		\end{align*}
		Then the embedding $id:\ell_{p_n}\hookrightarrow \ell_{q_n}$ is finitely strictly singular. Moreover $b_n(id)\le C n^{-\beta}$ for some positive constants $C,\beta$.
	\end{thm}
	\begin{proof}
		By the assumtion we have $K>0$ such that for all $n$
		\begin{align*}
			&\frac{1}{\frac{1}{p_n}-\frac{1}{q_n}}\le K.
		\end{align*}
		It gives
		\begin{align*}
			&0<\alpha:=\frac{p_-^2}{K}\le\frac{p_n q_n}{K}\le q_n-p_n.
		\end{align*}
		So $q_n-p_n$ is away from zero.
		Moreover
		\begin{align*}
			&p_n\le\frac{p_nq_n}{q_n-p_n}=\frac{1}{\frac{1}{p_n}-\frac{1}{q_n}}\le K
		\end{align*}
		and $p_n$ is bounded.
		
		Now we have  $p_-\le p_n\le p_n+\alpha\le q_n$. Moreover $p_n+\alpha$ is bounded. By Lemma \ref{kdwckdcckvckk} we know that the embedding $id_{\alpha}:\ell_{p_n}\hookrightarrow\ell_{p_n+\alpha}$ is strictly singular with $b_n(id_{\alpha})\le n^{-\frac{\alpha}{(q_+-\alpha)q_+}}$ and the embedding $\ell_{p_n}\hookrightarrow\ell_{q_n}$ is strictly singular as a composition of the embeddings $\ell_{p_n}\hookrightarrow\ell_{p_n+\alpha}$ and $\ell_{p_n+\alpha}\hookrightarrow\ell_{q_n}$. The behavior of $b_n(id)$ follows from the composition of operators.
	\end{proof}
	
	In what follows, in this section,  we assume $0<\min(p_-,q_-)\le  p_n,q_n\le \infty$ for all $n$. Set $A=\{n; p_n<q_n\}$. 
	
	Then by Theorem \ref{eflkkbtnbk} and Theorem \ref{xkclvlksjkjk} we know that
	$id: \ell_{q_n}\hookrightarrow\ell_{p_n}$ if and only if we have for some $0<c<1$
	\begin{align}
		&M:=\sum_{n\in\mathbb{A}}c^{\frac{1}{\frac{1}{p_n}-\frac{1}{q_n}}}<\infty.\label{emb}
	\end{align}
	
	\begin{lem}\label{sdpvjrghn}
		Assume \eqref{emb} and let $A$ be an infinite set. Then the embedding $id: \ell_{q_n}\hookrightarrow\ell_{p_n}$ is not strictly singular.
	\end{lem}
	\begin{proof}
		Set $X=\{\{a_n\}; a_n=0 \text{ for all }n\in \mathbb{N}\setminus A\}$. Then  $\dim X=\infty$, and norms on $\ell_{p_n}$ and $\ell_{p_n}$ on $X$ are equivalent by Theorem \ref{eflkkbtnbk}. Consequently the embedding is not strictly singular.
	\end{proof}
	
	\begin{lem}\label{lovfvjffjb}
		Assume \eqref{emb} and suppose that there exists an infinite set  $S\subset \mathbb{N}\setminus A$ such that for certain $0<c<1$
		\begin{align*}
			&\sum_{n\in S}c^{\frac{1}{\frac{1}{q_n}-\frac{1}{p_n}}}<\infty.
		\end{align*}
		Then the embedding $id:\ell_{q_n}\hookrightarrow\ell_{p_n}$ is not strictly singular.
	\end{lem}
	\begin{proof}
		We can see that on $S$, $id$ has inverse.
	\end{proof}
	
	\begin{lem}\label{fbvjfibjfibj}		Assume  \eqref{emb} and suppose that either $A$ is infinite or
		\begin{align*}
			&\limsup_{n\in \mathbb{N}\setminus A}\sum_{n\in S}{\frac{1}{\frac{1}{q_n}-\frac{1}{p_n}}}=\infty.
		\end{align*}
		Then the embedding $id:\ell_{q_n}\hookrightarrow\ell_{p_n}$ is not strictly singular.
	\end{lem}
	\begin{proof}
		Follows instantly from the above Lemma.
	\end{proof}
	
	\begin{lem}\label{fvljfiopvgjgj}
		Let $0< \min(p_-, q_-)\le  p_n,q_n\le \infty$ for all $n$. Set $A=\{n; p_n<q_n\}$ and suppose \eqref{emb}. Assume that  $A$ is finite and
		\begin{align*}
			&\limsup_{n\in \mathbb{N}\setminus A}\sum_{n\in S}{\frac{1}{\frac{1}{q_n}-\frac{1}{p_n}}}<\infty.
		\end{align*}
		Then the embedding $id:\ell_{q_n}\hookrightarrow\ell_{p_n}$ is strictly singular. Moreover $b_n(id)\le Cn^{-\beta} $ for some positive $C,\beta$.
	\end{lem}
	\begin{proof}
		Let $A$ has $k$ elements. Without lose of generality we can assume $A=\{1,2,\dots,k\}$. Recall
		\begin{align*}
			&b_n=\sup_{E\subset\ell_{q_n},\dim(E)=n}\inf_{0\neq a\in E} \frac{\|a\|_{p_n}}{\|a\|_{q_n}}.
		\end{align*}
		Fix $E\subset\ell_{q_n}$, $\dim(E)=n>k$. Take a basis $(b^{(1)},b^{(2)},\dots,b^{(n)})$ of $E$. Rewrite this basis into an infinite matrix
		\begin{align*}
			\begin{pmatrix}
				b^{(1)}_1&b^{(1)}_2&b^{(1)}_3&b^{(1)}_4&\dots\\
				b^{(2)}_1&b^{(2)}_2&b^{(2)}_3&b^{(2)}_4&\dots\\
				\vdots&&&&\\
				b^{(n)}_1&b^{(n)}_2&b^{(n)}_3&b^{(n)}_4&\dots
			\end{pmatrix}.
		\end{align*}
		Using the Gaussian elimination process we can find another basis $(a^{(1)},a^{(2)},\dots,a^{(n)})$ of $E$ such that $a^{(i)}_j=0$ for all $k<i\le n,\ 1\le j\le k$. After rewriting it into a matrix we have
		\begin{align*}
			\begin{pmatrix}
				a^{(1)}_1&a^{(1)}_2&a^{(1)}_3&\dots&a^{(1)}_k&a^{(1)}_{k+1}&a^{(1)}_{k+2}&\dots\\
				a^{(2)}_1&a^{(2)}_2&a^{(2)}_3&\dots&a^{(2)}_k&a^{(2)}_{k+1}&a^{(2)}_{k+2}&\dots\\
				\vdots&&&&&&&\\
				a^{(k)}_1&a^{(k)}_2&a^{(k)}_3&\dots&a^{(k)}_k&a^{(k)}_{k+1}&a^{(k)}_{k+2}&\dots\\
				0&0&0&\dots&0&a^{(k+1)}_{k+1}&a^{(k+1)}_{k+2}&\dots\\
				0&0&0&\dots&0&a^{(k+2)}_{k+1}&a^{(k+2)}_{k+2}&\dots\\
				\vdots&&&&&&&\\
				0&0&0&\dots&0&a^{(n)}_{k+1}&a^{(n)}_{k+2}&\dots
			\end{pmatrix}.
		\end{align*}
		Set $F=\spa(a^{k+1},a^{k+2},\dots, a^{n}$. Then $\dim(F)=n-k,\ F\subset E$. By Theorem \ref{nvjgkdocnvkvevjfe} we obtain
		\begin{align*}
			&\inf_{0\neq a\in F}\frac{\|a\|_{p_n}}{\|a\|_{q_n}}\le \sup_{F\subset\ell_{q_n}, \dim(F)=n-k}\inf_{0\neq a\in F}\frac{\|a\|_{p_n}}{\|a\|_{q_n}}\\
			&=b_n\big(Id:\ell_{q_n}(\mathbb{N}\setminus A)\rightarrow \ell_{p_n}(\mathbb{N}\setminus A)\big)\le C(n-k)^{-\beta}.
		\end{align*}
		Taking supremum over all $E$ we finally obtain
		\begin{align*}
			&\inf_{0\neq a\in F}\frac{\|a\|_{p_n}}{\|a\|_{q_n}}\le \sup_{F\subset\ell_{q_n}, \dim(F)=n-k}\inf_{0\neq a\in F}\frac{\|a\|_{p_n}}{\|a\|_{q_n}}, \qquad 	\mbox{ and } \\
			&b_n=\sup_{E\subset\ell_{q_n},\dim(E)=n}\inf_{0\neq a\in E} \frac{\|a\|_{p_n}}{\|a\|_{q_n}}\le C(n-k)^{-\beta}\le \widetilde{C}n^{-\beta}.
		\end{align*}
	\end{proof}
	
	\section{Appendix}
	 Here we provide some basic facts about quasi-Banach, quasi-modular and quasi-norm $\l_{p_n}$ spaces which we did not find in the standard literature. These results are mentioned here for the sake of completeness and convenience of readers.

	\begin{thm}
		Let $X,Y$ be quasi-Banach spaces and $T:X\rightarrow Y$ be linear bounded. If $T$ is finitely strictly singular then $T$ is strictly and singular.
	\end{thm}
	\begin{proof}
		Assume that $T$ is finitely strictly singular. Let $E\subset X$, $\dim(E)=\infty$. Choose $E_n\subset E$, $\dim(E)=n$. By the assumption we can find $n_k$ and $x_{n_k}\in E_{n_k}$ such that $\|x_{n_k}\|_X=1$ and $\|T(x_{n_k})\|_Y\le \frac{1}{k}$. Thus we have found a sequence $x_{n_k}\in E$, $\|x_{n_k}\|_X=1$ and $\|T(x_{n_k})\|_Y\le \frac{1}{k}$. So the inverse operator $T^{-1}$ cannot be bounded.
	\end{proof}
	
	\begin{thm}
		Let $X,Y$ be quasi-Banach spaces and $T:X\rightarrow Y$ be linear and  bounded. Then $T$ is finitely strictly singular if and only if $b_n(T)\rightarrow 0$.
	\end{thm}
	\begin{proof}
		Assume first that $b_n(T)\rightarrow 0$. Let $\varepsilon>0$. Find $n_0$ such that for $n>n_0$
		\begin{align*}
			&b_n(T)=\sup_{E\subset X, \dim(E)=n}\inf_{u\in E, \|u\|_X=1}\|Tu\|_Y\le \varepsilon.
		\end{align*}
		Let $E\subset X$, $\dim(E)=n$. By the definition of $b_n(T)$ we know that there is $0\neq x$ with $\frac{\|T(x)\|_Y}{\|x\|_X}\le \varepsilon$. Setting $z=\frac{x}{\|x\|_X}\in E$ we have $z\in E,\ \|z\|_X=1$ and $\|T(z)\|_Y\le\varepsilon$. So $T$ is finitely strictly singular.

		Assume now that $T$ is finitely strictly singular. By the definition we know that for each $\varepsilon>0$ there exists $n_0\in\mathbb{N}$ such that for all $n\ge n_0$ and for all subspaces $E\subset X$, $\dim(E)= n$ we can find $x\in E$ with $\|x\|_X=1$ and $\|T(x)\|_Y\le\varepsilon$. Choose $\varepsilon>0$ and find $n_0$ from the definition of finitely strictly singularity.
		We estimate $b_n(T)$. Take $E\subset X$, $\dim(E)=n>n_0$. Then there is $x\in E$ such that $\|x\|_X=1$ and $\|T(x)\|_Y\le\varepsilon$. So $\inf\limits_{0\neq x\in E}\frac{\|T(x)\|_Y}{\|x\|_X}\le\varepsilon$ and consequently
		\begin{align*}
			&b_n(T)=\sup_{E\subset X, \dim(E)=n}\inf_{x\in E, \|x\|_X=1}\|T(x)\|_Y\le \varepsilon.
		\end{align*}
		Thus $b_n(T)\rightarrow 0$.
	\end{proof}

	\begin{defi}
		Let $0<p_-\le p_n\le \infty$. For $x,y\in\ell_{p_n}$ we define
		\begin{align*}
			&\varrho_{p_n}(x,y):=\sum_{p_n<1}|x_n-y_n|^{p_n}+\|(x-y)\chi_{1\le p_n\le\infty}\|_{p_n}.
		\end{align*}
	\end{defi}

	\begin{lem}
		The function $\varrho_{p_n}$ is a metric on $\ell_{p_n}$.
	\end{lem}
	\begin{proof}
		It suffices to prove the triangle inequality only for the part $"p_n<1"$. Since $|a+b|^p\le |a|^p+|b|^p$ for $0<p<1$ we have
		\begin{align*}
			&\sum_{p_n<1}|x_n-y_n|^{p_n}\le\sum_{p_n<1}|x_n-z_n|^{p_n}+\sum_{p_n<1}|z_n-y_n|^{p_n},
		\end{align*}
	which conclude the proof.
	\end{proof}
	
%
%

Given a sequence $p_n$ with $0< p_n\le \infty$ we
set $B_{p_n}(x,\varepsilon):=\{y;\|x-y\|_{p_n}\}<\varepsilon$; $U_{p_n}(x,\varepsilon):=\{y;\varrho_{p_n}(x,y)\}<\varepsilon$.
\begin{lem}
Let $0<  p_n<1$.
Assume $0< \varepsilon <1$ and $y\in B_{p_n}(0,\varepsilon)$. Then there exists $\delta>0$ such that $U_{p_n}(y,\delta)\subset B_{p_n}(x,\varepsilon)$.
\end{lem}
\begin{proof}
We fix $y\in B_{p_n}(0,\varepsilon)$. Then
\begin{align*}
&\alpha:=\sum_{n=1}^\infty\Big|\frac{y_n}{\varepsilon}\Big|^{p_n}<1.
\end{align*}
Choose $\delta=\varepsilon(1-\alpha)$. We take $z\in U_{p_n}(y,\delta)$, i.e.
\begin{align*}
&\sum_{n=1}^\infty|z_n-y_n|^{p_n}<\delta.
\end{align*}
Then
\begin{align*}
&\sum_{n=1}^\infty\Big|\frac{z_n}{\varepsilon}\Big|^{p_n}\le \sum_{n=1}^\infty\Big|\frac{z_n-y_n}{\varepsilon}\Big|^{p_n}+\sum_{n=1}^\infty\Big|\frac{y_n}{\varepsilon}\Big|^{p_n}\\
&\le \frac{1}{\varepsilon}\sum_{n=1}^\infty|z_n-y_n|^{p_n}+\sum_{n=1}^\infty\Big|\frac{y_n}{\varepsilon}\Big|^{p_n}<\frac{1}{\varepsilon}\delta+\alpha=1.
\end{align*}
So $z\in B_{p_n}(0,\varepsilon)$.
\end{proof}

\begin{lem}
Let $0<p_-\le p_n<1$.
Assume $0< \varepsilon <1$ and $y\in U_{p_n}(0,\varepsilon)$. Then there exists $\delta>0$ such that $B_{p_n}(y,\delta)\subset B_{p_n}(x,\varepsilon)$.
\end{lem}
\begin{proof}
We fix $y\in U_{p_n}(0,\varepsilon)$. Then
\begin{align*}
&\alpha:=\sum_{n=1}^\infty|y_n|^{p_n}<\varepsilon.
\end{align*}
Choose $\delta=(\varepsilon-\alpha)^{1/{p_-}}$. We take $z\in B_{p_n}(y,\delta)$, i.e.
\begin{align*}
&\sum_{n=1}^\infty\Big|\frac{z_n-y_n}{\delta}\Big|^{p_n}<1.
\end{align*}
Then
\begin{align*}
&1>\sum_{n=1}^\infty\Big|\frac{z_n-y_n}{\delta}\Big|^{p_n}\ge \sum_{n=1}^\infty\frac{|z_n-y_n|^{p_n}}{\delta^{p_-}}
\end{align*}
which implies
\begin{align*}
&\sum_{n=1}^\infty|z_n-y_n|^{p_n}\le \delta^{p_-}.
\end{align*}
Further
\begin{align*}
&\sum_{n=1}^\infty|z_n|^{p_n}\le\sum_{n=1}^\infty|z_n-y_n|^{p_n}+\sum_{n=1}^\infty|y_n|^{p_n}< \delta^{p_-}+\alpha=\varepsilon.
\end{align*}
So $z\in U_{p_n}(0,\varepsilon)$.
\end{proof}

As an easy consequence we obtain the following theorem.
\begin{thm}
Let $0<p_-\le p_n\le \infty$. Then the topologies given by the metric $\varrho_{p_n}$ and the quasi-norm $\|.\|_{p_n}$ coincide.
\end{thm}
\begin{proof} When  $0<p_- \le p_n <1 $ then the statement follows from two previous Lemmas. In the case $1\le p_n \le \infty$ we can use that $\varrho_{p_n}(x,y)=\|f-g\|_{p_n}$. From these two observation is quite straightforward to establish the case  $0<p_-\le p_n\le \infty$.
\end{proof}

\begin{lem}[Analogy of Riesz lemma] \label{wdklvnwklbknkl}
Let $(X,\|.\|)$ be a quasi-normed space and $L\varsubsetneqq X$ be a closed subspace. Then for each $\varepsilon>0$ there exists $z\in X$, $\|z\|=1$, such that for all $y\in L$ we have $\|z-y\|>1-\varepsilon$.
\end{lem}
\begin{proof}
Fix $u\in X\setminus L$. First we verify that there is $\delta>0$ such that for all $y\in L$ it is $\|u-y\|\ge \delta$. If no then we can find a sequence $y_n\in L$ with $\|y-y_n\|\rightarrow 0$. Since $L$ is closed we have $y\in L$ which is a contradiction. Denote $d=:\inf\{\|u-y\|;y\in L\}>0$.

Choose $0<\varepsilon<1$ and $\eta<\frac{\varepsilon d}{1-\varepsilon}$. Due to the definition of $d$ we can find $v\in L$ such that $\|u-v\|<d+\eta$.
Set $z=\frac{u-v}{\|u-v\|}$. Clearly $\|z\|=1$.  Since $v+y\|u-v\|\in L$ we have for all $y\in L$
\begin{align*}
&\|z-y\|=\Big\|y-\frac{u-v}{\|u-v\|}\Big\|=\frac{\|y\|u-v\|-u+v\|}{\|u-v\|}\\
&=\frac{\|(v+y\|u-v\|)-u\|}{\|u-v\|}>\frac{d}{d+\eta}>1-\varepsilon.
\end{align*}
\end{proof}

\begin{lem}
Let $0<p_-\le p_n\le \infty$. Then $K\subset \ell_{p_n}$ is compact (in the sense of metric $\varrho_{p_n}$) if and only if from each covering $K\subset \bigcup_{i\in I}B_{p_n}(x_i,\varepsilon_i)$ we can take a finite subset $F\subset I$ such that $K\subset \bigcup_{i\in F}B_{p_n}(x_i,\varepsilon_i)$.
\end{lem}
\begin{proof}
The proof immediately follow s from the fact that topologies given by $\varrho_{p_n}$ and $\|.\|_{p_n}$ coincide.
\end{proof}

\begin{thm}
Let $0<\min\{p_-,q_-\}\le p_n,q_n\le \infty$ and assume that $id:\ell_{q_n}\rightarrow\ell_{p_n}$ is compact. Then $b_n(id)\rightarrow 0$.
\end{thm}
\begin{proof}
Assume $\limsup_{n\rightarrow\infty}b_n(id)\ge \delta>0$. Choose $\varepsilon<\frac{\delta}{2T}$ where $T$ is the constant of triangle inequality in $\ell_{p_n}$. Since $B_{p_n}(0,1)$ is pre-compact in $\ell_{p_n}$ we have a finite number of elements $b_1,b_2,\dots,b_{n_0}\in \ell_{p_n}$ such that $\bigcup_{i=1}^{n_0}B_{p_n}(b_i,\varepsilon)\supset \overline{B_{p_n}(0,1)}$.

Fix now $n>n_0$, $b_n(id)\ge\delta$. Take $0<\eta<1-\frac{2T}{\delta}\varepsilon$. By Lemma \ref{wdklvnwklbknkl} we have $a_1,a_2,\dots,a_n\in \ell_{p_n}$ such that $\|a_i\|_{q_n}=1$ and $\|a_i-a_j\|_{q_n}>1-\eta$. By the definition of $b_n(id)$ and the fact $b_n(id)\ge\delta$ we have
\begin{align*}
&\|a_i\|_{p_n}\ge \delta, \ \|a_i-a_j\|_{p_n}\ge\delta(1-\eta).
\end{align*}
We show that the quasi-balls $B_{p_n}(a_i,\varepsilon)$ are pairwise disjoint. Indeed  taking $x,y$ with $\|x-a_i\|_{p_n}<\varepsilon$, $\|y-a_j\|_{p_n}<\varepsilon$ we have
\begin{align*}
&\|x-y\|_{p_n}\ge \frac{\|a_i-a_j\|_{p_n}}{T}-\|x-a_i\|_{p_n}-\|a_j-y\|_{p_n}\ge \frac{\delta(1-\eta)}{T}-2\varepsilon>0.
\end{align*}
We see from this fact that minimal number of quasi-balls covering $\overline{B_{p_n}(0,1)}$ is at least $n$. It is a contradiction with the fact that we covered   $\overline{B_{p_n}(0,1)}$ by $n_0$ quasi-balls. So $\lim_{n\rightarrow\infty}b_n(id)=0$.
\end{proof}


\begin{thebibliography}{0}
		
		
		\bibitem{AA book}
		\by{\name{Y. A.}{Abramovich, }\name{ C. D.}{Aliprantis}}
		\paper{An invitation to operator theory, Graduate Studies in Mathematics}
		\jour{American Mathematical Society, Providence, RI,}
		\vol{50,}
		\yr{2002}
		\pages{xiv+530, ISBN: 0-8218-2146-6}
		\endpaper
		
		\bibitem{BS}
		\by{\name{C.}{Bennet}\et\name{R.}{Sharpley}}
		\book{Interpolations
			of operators}
		\publ{Pure and Apl. Math., vol.~129, Academic Press, New York, 1988}
		\endbook
		
	\bibitem{EN}
	\by{\name{D.E.}{Edmunds}\et\name{A.}{Nekvinda}}
	\paper{Averaging operators on {$l^{\{p_n\}}$} and {$L^{p(x)}$}}
	\jour{Math. Inequal. Appl.}
	\vol{5, \rm{no. 2}}
	\yr{2002}
	\pages{235--246}
	\endpaper	
		
		
		\bibitem{KR}
		\by{\name{O.}{Kov\'a\v cik}\et\name{J.}{R\'akosn\'ik}}
		\paper{On spaces $L^{p(x)}(\Omega)$ and $W^{k,p(x)}$}
		\jour{Czechoslovak Math. J.}
		\vol{41(116), \rm{no. 4}}
		\yr{1991}
		\pages{592–-618}
		\endpaper
		
		\bibitem{LM}
		\by{\name{J.}{Lang}\et\name{V.}{Musil}}
		\paper{Strict $s$-numbers of non-compact Sobolev embeddings into continuous functions}
		\jour{Constr. Approx.}
		\vol{50, \rm{no. 2}}
		\yr{2019}
		\pages{271–291}
		\endpaper
		
		
		\bibitem{LRP}
		\by{\name{P.}{Lef\`evre,} \name{L.}{Rodríguez-Piazza,}}
		\paper{Finitely strictly singular operators in harmonic analysis and function theory}
		\jour{Adv. Math.}
		\vol{255}
		\yr{2014}
		\pages{119-152, DOI 10.1016/j.aim.2013.12.034}
		\endpaper
		
		
			\bibitem{Lor}
		\by{\name{G.G.}{Lorentz,}}
		\paper{Lower bounds for the degree of approximation}
		\jour{Trans. Amer. Math. Soc.}
		\vol{97}
		\yr{1960}
		\pages{25-34}
		\endpaper
		
		\bibitem{N}
		\by{\name{A.}{Nekvinda}} \paper{Imbeddings between discrete
			weighted Lebesgue spaces with variable exponents} \jour{Math.
			Inequal. Appl.} \vol{10, \rm{no. 1,}} \yr{2007} \pages{165--172}
		\endpaper
		
		\bibitem{NP}
		\by{\name{A.}{Nekvinda}\et\name{D.}{Pe\v sa}}
		\paper{On the properties of quasi-Banach function spaces}
		\jour{arXiv:2004.09435 [math.FA] }
		\endprep
		
			\bibitem{Mil}
		\by{\name{V.D.}{Milman}}
		\paper{Operators of class $C_0$ and $C^*_0$}
		\jour{Teor. Funktsii Funktsional. Anal. i Prilozen.}
		\vol{10}
		\yr{1970}
		\pages{15-18, (Russian)}
		\endpaper
		
		
		
		
		\bibitem{Or}
		\by{\name{W.}{Orlicz}}
		\paper{Über konjugierte Exponentenfolgen.}
		\jour{Studia Mathematica}
		\vol{3}
		\yr{1931}
		\pages{200–211}
		\endpaper
		
		

	\bibitem{Pitt}
\by{\name{H.R.}{Pitt}}
\paper{A note on bilinear forms.}
\jour{J. Lond. Math. Soc.}
\vol{11(1)}
\yr{1932}
\pages{174–180}
	
\endpaper

		
			\bibitem{Pl}
		\by{\name{A.}{Plichko}}
		\paper{Superstrictly singular and superstrictly cosingular operators}
		\jour{Functional analysis and its applications, North-Holland Math. Stud., Elsevier Sci. B. V., Amsterdam,}
		\vol{197}
		\yr{2004}
		\pages{239–255}
		\endpaper
		
			\bibitem{Ti}
		\by{\name{V.M.}{Tikhomirov}}
		\book{Some question in approximation theory}
		\publ{Izdat. Moskov. Univ., Moscow, 1976 (Russian)}
		\endbook
		
		
	\end{thebibliography}
\end{document}